\documentclass[reqno]{amsart}
\usepackage{amsfonts}
\usepackage{amsmath}
\usepackage{mathrsfs}
\usepackage{amsfonts,amssymb,amsmath,amsthm,setspace}
\usepackage{hyperref}
\usepackage{enumitem}
\allowdisplaybreaks
\usepackage{amsfonts}
\usepackage{amsmath}
\usepackage{mathrsfs}
\usepackage{amsfonts,amssymb,amsmath,amsthm}
\usepackage{hyperref}
\usepackage{enumitem}
\usepackage{xcolor}
\usepackage[numbers]{natbib}
\numberwithin{equation}{section}

\newtheorem{theorem}{Theorem}[section]
\newtheorem{lemma}{Lemma}[section] 
\usepackage{mathtools}
\usepackage{multirow,array}
\usepackage{graphicx}
\usepackage{subfigure}
\usepackage{epstopdf}
\usepackage[numbers]{natbib}
\usepackage{float} %%% to use [H], [ht] etc. to fix places of figures
\usepackage{color, xcolor}
\allowdisplaybreaks

\begin{document}

\title[Brezis-Nirenberg type problem for fractional
sub-Laplacian]{Brezis-Nirenberg type problem for fractional
sub-Laplacian on the Heisenberg group}

%%%%% authors:
\author[V. Y. Naik and G. Dwivedi]{
        Vikram Yallapa Naik$^1$ %%% 1 %%% V. Kiryakova %% ORCID ID: 0000-0002-4591-6958 %% pls. add IF available %%
\and
        Gaurav Dwivedi$^2$ %%% 2 %%% S. Kiryakov %% ORCID ID: N/A %%%
 }
\address{Department of Mathematics, Birla Institute of Technology and Science Pilani, Pilani Campus, Pilani-333031 (Rajasthan), India}
\email{$^1$vikramnaik154@gmail.com, $^2$gaurav.dwivedi@pilani.bits-pilani.ac.in}

%%%%%%%%%%%%%%%%%%%%%%%%%%%%%%%%%%%%%%%%%%%%%%%%%%%%%%
\begin{abstract}
% {\color{blue}
% Text of the abstract. \dots }
In this paper, we show the existence of a weak solution for a fractional
 sub-Laplace equation involving a term with the critical Sobolev exponent, namely,  
\begin{align*}
    (-\Delta_\mathbb{H})^su - \lambda u &= |u|^{Q^*_s -2}u \text{ in } \Omega, \\
    u &= 0 \text{ in } \mathbb{H}^N \setminus \Omega,  
\end{align*}
where $\Omega \subseteq \mathbb{H}^N$ is bounded and has continuous boundary, $(-\Delta_\mathbb{H})^s$ is the horizontal fractional Laplacian, $s \in (0,1), \lambda > 0,$ and $Q^*_s=\frac{2Q}{Q-2s}$ is the Sobolev critical exponent. This problem is motivated by the celebrated Brezis-Nirenberg problem \cite{brezis1983positive}.

\end{abstract} %%%%%%%%%%%%%
\maketitle
%%%%%%%% begin papers' body %%%%%%%%%%%%%%%%%%%%%%%%%%%%%

%%%%%%%%%%%% Section 1 %%%%%%%%%%%%%%%%%%%%%%%%%%
\section{Introduction} \label{sec:1}

\setcounter{section}{1} \setcounter{equation}{0} %% to have proper 2-digits numbers of eqs
%% Note that this style produces 1-digit numbering of definitons, statements, exmaples, etc.

In this article, we consider the problem  
\begin{align}\label{eqn1.1}
\begin{split}
    (-\Delta_\mathbb{H})^su - \lambda u &= |u|^{Q^*_s -2}u \text{ in } \Omega, \\
    u &= 0 \text{ in } \mathbb{H}^N \setminus \Omega,
\end{split}    
\end{align}
where $\Omega \subseteq \mathbb{H}^N$ is bounded and has continuous boundary, $(-\Delta_\mathbb{H})^s$ is the horizontal fractional Laplacian, $s \in (0,1), \lambda > 0,$ and $Q^*_s=\frac{2Q}{Q-2s}$ is the critical Sobolev exponent.

The Problem \eqref{eqn1.1} is commonly referred to as the Brezis-Nirenberg problem, named after the celebrated article of Brezis and Nirenberg \cite{brezis1983positive}, where the authors proved the existence of positive solutions for the problem
\begin{align} \label{eqn1.2}
\begin{split}
    -\Delta u - \lambda u &= |u|^{2^* -2}u \text{ in } \Omega, \\
    u &= 0 \text{ on } \partial \Omega.
\end{split} 
\end{align}
Here $\Omega \subset \mathbb{R}^N$ is bounded, $\lambda > 0,$ and $2^*=\frac{2N}{N-2}$ is the critical Sobolev exponent. Brezis and Nirenberg proved the following: 
\begin{enumerate}
    \item[(i)] If $N\geq 4,$ the Problem \eqref{eqn1.2} 
 has a positive solution for every \( \lambda \in (0, \lambda_1) \), where \(\lambda_1\) denotes the first eigenvalue of \(-\Delta\) with Dirichlet boundary condition; moreover, it has no solution if $\lambda \leq 0$ and \( \Omega \) is star-shaped.
\item[(ii)] If \( N = 3 \) and   \( \Omega \) is a ball, then \eqref{eqn1.2}  has a solution if and only if \( \lambda \in \left(\frac{1}{4}\lambda_1, \lambda_1\right) \).
\end{enumerate}

The equations of the type \eqref{eqn1.1} in the case of fractional Laplacian $(-\Delta )^s$ in the Euclidean domain $\mathbb{R}^N$ were first studied by Servadei and Valdinoci \cite{servadei2015brezis}. They proved that the problem 
\begin{align}\label{frac_model}
\begin{split}
    (-\Delta)^su - \lambda u &= |u|^{2^*_s -2}u \text{ in } \Omega\subseteq \mathbb{R}^N, \\
    u &= 0 \text{ in } \mathbb{R}^N \setminus \Omega,
\end{split}    
\end{align}
admits a non-trivial solution for any $\lambda\in (0, \lambda_{1,s})$, provided $N \geq 4s,$ where $\lambda_{1,s}$ is the first eigenvalue of the fractional Laplacian with nonlocal Dirichlet boundary condition. The results of \cite{servadei2015brezis}  were further extended to fractional $p$-Laplacian by Mosconi et al. \cite{mosconi2016brezis}, and Mahwin and Molica Bisci \cite{mawhin2017brezis}.  For some further existence results involving critical fractional elliptic problems, we refer to \cite{barrios2012some,barrios2015critical,bonder2018concentration,brasco2018optimal,cantizano2019three,mukherjee2016fractional,pucci2023brezis,ros2015nonexistence,shang2014positive,xiang2016existence}. 

% In \cite{capozzi1985existence}, Capozzi et al. have proved existence results for $N\geq4$ and any $\lambda > 0.$ In \cite{schechter2010brezis}, Schechter et al. have proved the existence of infinitely many sign-changing solutions except zero for \eqref{eqn1.2} for $N \geq 7$ if $\lambda \geq \lambda_1,$ using Morse theory. 
 
% In \cite{gao2018brezis}, Gao et al. have proved the existence results for the Choquard-type equation
% \begin{align*}
% -\Delta u &= \left(\int\limits_{\Omega} \frac{|u|^{2^*_\alpha}}{|x - y|^{\alpha}} dy \right)  |u|^{2^*_\alpha - 2} u + \lambda u \text{ in } \Omega, \\
%     u &= 0 \text{ on } \partial \Omega,
% \end{align*}
% where, $\Omega \subset \mathbb{R}^N$ is bounded, $\lambda > 0, \alpha \in (0,N)$ and $2^*_\alpha=\frac{2N-\alpha}{N-2}$ is the critical exponent with respect to Hardy-Littlewood-Sobolev inequality and Sobolev embedding. In this article, they have proved the existence of a solution for all $\lambda > 0.$

% In \cite{servadei2015brezis}, the authors have proved the existence of the non-local counterpart, i.e. fractional version of \eqref{eqn1.2}. In \cite{mukherjee2016fractional}, authors have dealt with the Brezis-Nirenberg type problem for the Choquard type nonlinearity. They have proved the existence result for all $\lambda > 0.$

Brezis-Nirenberg problem -- for non-fractional case -- has also been studied in non-Euclidean settings such as the Riemannian manifold \cite{bandle2002brezis}, Grushin manifold \cite{alves2024brezis}, Heisenberg group \cite{citti1995semilinear,goel2020existence}, hyperbolic spaces \cite{carriao2019brezis}, Carnot group \cite{loiudice2007semilinear,loiudice2015critical} etc. Y. Han~\cite{han2020integral} studied the Brezis–Nirenberg problem involving an integral operator on the Heisenberg group.
% In \cite{goel2020existence}, Goel et al. have proved the existence result for the Brezis-Nirenberg problem with the Choquard term in this setting for all $\lambda > 0.$

Recently, there has been noticeable work done on the problems involving the fractional sub-Laplace operator on the Heisenberg group. Mallick et al. \citep{mallick2018hardy} derived the Hardy inequality for fractional Sobolev spaces on the Heisenberg group and further proved Sobolev and Morrey type embeddings in this setting. Wang et al. \citep{wang2020properties} established the symmetry, 
monotonicity and Liouville property of solutions to  fractional $p$-subLaplace equations on the Heisenberg group. Zhou et al. \citep{zhou2022kirchhoff} proved the existence of weak solution to the fractional Kirchhoff type problem  with subcritical nonlinearity on the Heisenberg group. For some regularity results on problems involving fractional p-subLaplace operator, we refer to \cite{manfredini2023holder,palatucci2023nonlinear}.  Piccinini \citep{piccinini2022obstacle} obtained the existence and uniqueness results the obstacle problem for fractional p-subLaplace operator and  also discussed Perron Method for nonlinear fractional equations in the Heisenberg group. Ghosh et al. \citep{ghosh2024compact} studied an eigenvalue problem for the fractional $p$-sub-Laplacian over the fractional Folland–Stein–Sobolev spaces on stratified Lie groups and discussed a singular Brezis-Nirenberg type problem for fractional sub-Laplacian. Ghosh et al.~\cite{ghosh2024subelliptic} also studied the Brezis-Nirenberg problem p-sub-Laplacian on stratified Lie groups.

In view of recent developments on the fractional Laplacian on the Heisenberg group, we look to extend the results in \cite{servadei2015brezis} in the Heisenberg group setting. The main result of this article is as follows:
\begin{theorem}\label{thm1}
    Let $ \lambda_1$ be the first eigenvalue of the fractional sub-Laplace operator defined in \eqref{eigen}. Then there exists a non-trivial solution of \eqref{eqn1.1} if $0 < \lambda < \lambda_1.$  
\end{theorem}

The plan of this paper is as follows: In Section 2, we introduce some well-known definitions and results, and in Section 3, we provide proof of Theorem \ref{thm1}.

%%%%%%%%%%%%%%% Section 2 %%%%%%%%%%%%%%%%%%%%%%%
 \section{Preliminaries and Functional Space} \label{sec:2}

\setcounter{section}{2} \setcounter{equation}{0} %% to have proper 2-digits numbers of eqs
%% Note that this style produces 1-digit numbering of definitons, statements, exmaples, etc.

We briefly introduce the necessary terminologies in the Heisenberg group. The Heisenberg group is denoted by $\mathbb{H}^N=\left( \mathbb{R}^{2N+1}, \circ \right),$ where `$\circ$' denotes the group operation  defined as  
$$
    \xi\circ\xi'=(x+x',y+y',t+t'+2(x'\cdot y-y'\cdot x)),
$$
    for every $\xi=(x,y,t),\xi '=(x',y',t')\in \mathbb{H}^N.$

% $\xi^{-1}=-\xi$ is the inverse, and therefore
% $(\xi')^{-1}\circ {\xi}^{-1}=(\xi\circ\xi')^{-1} .$

The natural group of dilations on $\mathbb{H}^N$ is defined as
$\delta_s(\xi)=(sx,sy,s^2t),$ for every $s>0.$ Hence,
$\delta_s(\xi'\circ\xi) = \delta_s(\xi')\circ \delta_s(\xi)$ and $\delta_{s'}\delta_s(\xi)= \delta_{s's}(\xi), s, s'>0.$ It
can be seen that the Jacobian determinant of dilations
$\delta_s: \mathbb{H}^N\to \mathbb{H}^N$ is constant and equal to
$s^Q,$ for every $\xi=(x,y,t)\in \mathbb{H}^N.$ The natural number $Q=2N+2$ is
called the homogeneous dimension of $\mathbb{H}^N.$
The left invariant vector fields on $\mathbb{H}^N$ are given by
$$
T=\frac{\partial}{\partial t},  \
X_{j}=\frac{\partial}{\partial
x_{j}}+2y_{j}\frac{\partial}{\partial t},  \
Y_{j}=\frac{\partial}{\partial y_{j}}-2x_{j}\frac{\partial}{\partial
t}, \quad j=1,\ldots,N.
$$
The horizontal gradient is given by $$\nabla_\mathbb{H} u(\xi)=(X_{1}u(\xi),X_{2}u(\xi),\cdots,X_{N}u(\xi),Y_{1}u(\xi),Y_{2}u(\xi),\cdots,Y_{N}u(\xi)).$$ 
The subelliptic Laplacian on $\mathbb{H}^N$ is defined as $$\Delta_{\mathbb{H}}u(\xi)=\sum\limits_{j=1}^{N}X_{j}^{2}u(\xi)+Y_{j}^{2}u(\xi).$$
The homogeneous norm on $\mathbb{H}^N$ is defined as
$$|\xi|=|\xi|_{\mathbb{H}}=\left[(|x|^2+|y|^2)^2+t^2\right]^{\frac{1}{4}},  \  \mbox{for every}  \   \xi\in \mathbb{H}^N.$$ The homogeneous degree of the norm is $1,$ in terms of dilations.
In general, a homogeneous norm does not satisfy the triangle inequality. If $d_0$ is any homogeneous norm then, $d_0(\xi \circ \eta) \leq \tilde{c} (d_0(\xi) + d_0(\eta)),$ for $\xi, \eta \in \mathbb{H}^N.$  But when the homogeneous norm $d_0$ is \( |\cdot|_{\mathbb{H}} \), the constant \( \tilde{c} \) can be chosen equal to 1. For further details, we refer to \cite[Remark 2.5]{manfredini2023holder}.

The $\mathbb{R}^{2N+1}-$Lebesgue measure is the Haar measure on $\mathbb{H}^N$, which is invariant under the left translations of $\mathbb{H}^N$ and is $Q$-homogeneous in terms of dilations, where $Q=2N+2.$
Hence, the topological dimension $2N + 1$ of $\mathbb{H}^{N}$ is strictly less than its Hausdorff
dimension $Q.$ Next, $|\Omega|$ denotes the
$(2N + 1)$-dimensional Lebesgue measure of any measurable set $\Omega\subseteq
\mathbb{H}^{N}.$ Thus,
$$|\delta_{s}(\Omega)|=s^{Q}|\Omega|, s > 0,  \   d(\delta_{s}\xi)=s^{Q}d\xi,$$
and
$$
|B_r(\xi_{0})|=\alpha_{Q}r^{Q},  \  \mbox{where}  \   \alpha_{Q}=|B_1(0)|.$$
Here $B(\xi_{0},r)$ denotes the ball in $\mathbb{H}^N$ centred at $\xi_{0}$ with radius $r.$

The fractional sub-Laplacian is defined as
\begin{equation}\tag{P} \label{P}
     (-\Delta_{\mathbb{H}})^s u(\xi)  := C(N, s) \text{PV} \int\limits_{\mathbb{H}^N} \frac{ u(\xi) - u(\eta)}{|\eta^{-1} \circ \xi|^{Q+2s}}  d\eta, \quad \xi \in \mathbb{H}^N,
\end{equation} 
where PV denotes principal value.

Integro-differential operators in the form of \eqref{P}  arise as a generalization of the classical conformally invariant fractional sub-Laplacian $\left( -\Delta_{\mathbb{H}} \right)^s$ on $\mathbb{H}^N$, which was first introduced in \cite{branson2013moser} via the spectral formula
\[
\left( -\Delta_{\mathbb{H}} \right)^s = 2^s  |T|^s \frac{\Gamma\left(-\frac{1}{2} \Delta_{\mathbb{H}}|T|^{-1} +\frac{1+s}{2} \right)}{\Gamma\left(-\frac{1}{2} \Delta_{\mathbb{H}}|T|^{-1} +\frac{1-s}{2} \right)}\quad s \in (0, 1),
\]
with $\Gamma(x) := \int\limits_0^1 t^{x-1}e^{-t} dt$ being the Euler Gamma function, $\Delta_{\mathbb{H}}$ is the sub-Laplacian in $\mathbb{H}^N$, and $T = \partial_t$ is the vertical vector field. In \cite{roncal2016hardy}, it has been proved that \eqref{P} is the representation formula for the above spectral formula.

The fractional horizontal Sobolev space for a fixed $s \in (0,1)$ is defined as,
$$
    HW^{s,2}(\Omega)= \left\{ u \in L^2(\Omega) : \frac{|u(\xi)-u(\eta)|}{\left|\eta^{-1} \circ \xi \right|^{\frac{Q}{2} +s}} \in L^2(\Omega \times \Omega)\right\}.
$$
The norm in this space is given by,
$$
    \|u\|_{HW^{s,2}(\Omega)}=\left(\|u\|_{L^2(\Omega)}^2 + [u]^2_{HW^{s,2}(\Omega)} \right)^{\frac{1}{2}},
$$
where $[u]^2_{HW^{s,2}(\Omega)}=\iint\limits_{\Omega \ \Omega}\frac{|u(\eta) - u(\xi)|^2}{\left|\eta^{-1} \circ \xi \right|^{Q+2s}} \ d \xi d\eta,$ $\|u\|_{L^p(\Omega)}=\left(\int\limits_\Omega |u(\xi)|^p d\xi\right)^{\frac{1}{p}}.$ Let $\mathcal{O} = \Omega^c \times \Omega^c,$ and define $\mathcal{S} = \mathbb{H}^{N} \times \mathbb{H}^{N} \setminus \mathcal{O}.$
 Let $X$ be the space of Lebesgue measurable functions $u:\mathbb{H}^N\to\mathbb{R}$ such that their restriction to $\Omega$  belongs to $L^2(\Omega)$ and 
$$
 [u] = [u]_X = \left(\iint\limits_{\mathcal{S}}\frac{|u(\eta) - u(\xi)|^2}{\left|\eta^{-1} \circ \xi \right|^{Q+2s}} \ d\xi d\eta \right)^{\frac{1}{2}} < \infty.
$$

\noindent The space $X$ is equipped with the norm,
\[
 \|u\|=\|u\|_X=\left(\|u\|_{L^2(\Omega)}^2 + [u]_X^2\right)^{\frac{1}{2}}.
\]

\noindent Let $X_0 = \{ u \in X : u = 0 \text{ a.e. in } \mathbb{H}^N \setminus \Omega \}$. Then, we have the following Poincar\'{e} type inequality on $X_0:$
\begin{theorem}[Poincar\'e-type Inequality {\cite[p.~4210]{ghosh2024compact}}] There exists a constant \( C > 0 \) such that  
\[
\|u\|_{L^2(\Omega)} \leq C [u]_X, \quad \text{for all } u \in X_0.
\]  
In particular, \( [u]_X \) defines a norm on \( X_0 \).
\end{theorem}

\noindent Ghosh et al.~\cite{ghosh2024compact} define the space \( X^{s,2}_0(\Omega) \), where \( \Omega \) is a bounded open subset with continuous boundary of a stratified Lie group, as the closure of \( C_c^\infty(\Omega) \) with respect to the norm  
\[
\|u\|_{X^{s,2}_0(\Omega)} = \|u\|_{L^2(\Omega)} +  \left(\ \iint\limits_{\mathbb{H}^N \ \mathbb{H}^N} \frac{|u(\eta) - u(\xi)|^2}{\left|\eta^{-1} \circ \xi \right|^{Q + 2s}} \, d\xi \, d\eta \right)^{\frac{1}{2}}.
\]  
They further establish the following characterization:
\[
X^{s,2}_0(\Omega) = \left\{ u \in HW^{s,2}(\mathbb{H}^N) : u = 0 \text{ a.e. in } \mathbb{H}^N \setminus \Omega \right\}.
\]  
Therefore, the space \( X_0 \) introduced earlier coincides with \( X^{s,2}_0(\Omega) \). From~\cite[Theorem~7]{ghosh2024compact}, it follows that the space \( X_0 \) is continuously embedded in \( L^r(\Omega) \) for all \( 1 \leq r \leq Q_s^* \), and compactly embedded in \( L^r(\Omega) \) for all \( 1 \leq r < Q_s^* \).

\noindent We define, $$S_s = \inf\limits_{u \in X_0\backslash\{0\}} \frac{[u]^2}{\|u\|^2_{L^{Q^*_s}(\Omega)}}.$$
One can prove that $S_s$ is independent of $\Omega$ and is achieved only when $\Omega=\mathbb{H}^N$(see for instance \cite[Lemma 2.2]{mukherjee2016fractional}). The infimum  is achieved by the function (up to translation and dilation) 
\begin{equation} \label{eqn2}
    U(\xi) = U(x,y,t) = \frac{C}{\left( t^2 + (1 + |x|^2 + |y|^2)^2  \right)^{\frac{Q-2s}{4}}}.
\end{equation}
This was proved by Frank and Lieb~\cite{frank2012sharp}. One can also refer to \cite{frank2015extension,garofalo2024optimal}.

\noindent Let $A(u) = \iint\limits_{\mathcal{S}} \frac{|u(\xi) - u(\eta)|^2}{|\eta^{-1} \circ \xi|^{Q+2s}}  d\xi  d\eta.$ The first eigenvalue for the fractional sub-Laplace operator is given by 
\begin{equation}\label{eigen}
    \lambda_1 = \inf\{A(u) : \|u\|_{L^2(\Omega)} = 1, u \in X_0 \}.
\end{equation}
Hence, we have
$$
    \lambda_1 \int\limits_{\Omega} |u|^2  d \xi \leq  \iint\limits_{\mathcal{S}}\frac{|u(\eta) - u(\xi)|^2}{\left|\eta^{-1} \circ \xi \right|^{Q+2s}}  d \xi d\eta.
$$
For more details on eigenvalues of the fractional sub-Laplace operator, we refer to \cite{ghosh2024compact}.

We say that $u \in X_0$ is a weak solution of \eqref{eqn1.1} if
\small
\begin{equation} \label{ws}
    \iint\limits_S \frac{(u(\xi) - u(\eta))(v(\xi) - v(\eta))}{\left|\eta^{-1} \circ \xi \right|^{Q+2s}}  d \xi d \eta = \int\limits_{\Omega} |u(\xi)|^{Q_s^* - 2} u(\xi) v(\xi)  d \xi + \lambda \int\limits_{\Omega} u(\xi)v(\xi)  d\xi
\end{equation}
\normalsize
for every $v \in X_0$.
The energy functional corresponding to the problem \eqref{eqn1.1} is given by
\begin{equation} \label{eqn2.3}
    I_\lambda(u) = \frac{1}{2} \|u\|^2 - \frac{\lambda}{2} \int\limits_{\Omega} |u|^2  d \xi  -  \frac{1}{Q_s^*}\int\limits_{\Omega} |u|^{Q_s^*}  d \xi.
\end{equation}
One can verify that the critical points of $I_\lambda$ are the weak solutions to \eqref{eqn1.1}.

%%%%%%%%  Section 3 %%%%%%%%%%%%%%%%%%%%%%%%%%%%%%
\section{Proof of Theorem \ref{thm1}} \label{sec:3}

\setcounter{section}{3} \setcounter{equation}{0} 
To prove Theorem~\ref{thm1}, we establish a series of auxiliary lemmas. We begin by defining the quantity
$$
S_{s, \lambda} = \inf\limits_{u \in X_0 \setminus\{0\}} \frac{\iint\limits_{\mathcal{S}} \frac{|u(\xi) - u(\eta)|^2}{|\eta^{-1} \circ \xi|^{Q+2s}} d\xi d\eta - \lambda \int\limits_{\mathbb{H}^N} |u(\xi)|^2  d\xi }{\left( \int\limits_{\mathbb{H}^N} |u(\xi)|^{Q^*_s} d \xi \right)^{\frac{2}{Q^*_s}}}.
$$
We aim to show that this infimum is attained by some function \( u \in X_0 \).

Let \( U \) be the function given in~\eqref{eqn2}, and define
\begin{equation*}
    \overline{u}(\xi) = \frac{U(\xi)}{\|U\|_{L^{Q^*_s}(\mathbb{H}^N)}}.
\end{equation*}
By definition, we then have
\[
S_s = \iint\limits_{\mathcal{S}} \frac{|\overline{u}(\xi) - \overline{u}(\eta)|^2}{|\eta^{-1} \circ \xi|^{Q+2s}} \, d\xi \, d\eta.
\]

Next, we consider the following critical problem:
\begin{equation} \label{eqn3.2}
    (-\Delta_{\mathbb{H}})^s u = |u|^{Q^*_s - 2}u \quad \text{in } \mathbb{H}^N.
\end{equation}

 \begin{lemma}    $u^*(\xi) =  \overline{u}\left( \delta_{S_s^{-\frac{1}{2s}}}(\xi)\right)$ is a solution of Problem \eqref{eqn3.2} satisfying the property
$\int\limits_{\mathbb{H}^N} |u^*(\xi)|^{Q^*_s} d\xi = S_s^{\frac{Q}{2s}}.$
 \end{lemma}
\begin{proof}
     Since $U$ is an optimizer, there exists a Lagrange multiplier $\gamma \in \mathbb{R}$ such that
\begin{equation*}
\iint\limits_{\mathbb{H}^{N} \mathbb{H}^{N}} \frac{(\overline{u}(\xi) - \overline{u}(\eta))(v(\xi) - v(\eta))}{|\eta^{-1} \circ \xi|^{Q+2s}} d\xi d\eta  = \gamma \int\limits_{\mathbb{H}^N} |\overline{u}(\xi)|^{Q^*_s - 2} \overline{u}(\xi)v(\xi)  d\xi,
\end{equation*}
for any $v \in HW^{s,2}(\mathbb{H}^N)$.
Taking $v = \overline{u}$ as a test function, we get
\begin{equation*}
S_s = \iint\limits_{\mathbb{H}^{N} \ \mathbb{H}^{N}} \frac{|\overline{u}(\xi) - \overline{u}(\eta)|^2} {|\eta^{-1} \circ \xi|^{Q+2s}} d\xi d\eta = \gamma,
\end{equation*}
since $\|\overline{u}\|_{L^{Q^*_s}(\mathbb{H}^N)} = 1.$

Note that $u^* \in HW^{s,2}(\mathbb{H}^N) \setminus \{0\}$. 
Let $\delta_{S_s^{-\frac{1}{2s}}}(\xi) = \xi'$ and $\delta_{S_s^{-\frac{1}{2s}}}(\eta) = \eta'.$ So, $S_s^{-\frac{Q}{2s}} d \xi = d \xi' $ and $S_s^{-\frac{Q}{2s}} d \eta = d \eta'.$ Then, using these change of variables, for any $v \in HW^{s,2}(\mathbb{H}^N),$
\begin{align*}
    \iint\limits_{\mathbb{H}^{N} \ \mathbb{H}^{N}} &\frac{ (u^*(\xi) -  u^*(\eta))(v(\xi) - v(\eta) )}{|\eta^{-1} \circ \xi|^{Q+2s}} d\xi d\eta \\
&=\iint\limits_{\mathbb{H}^{N} \ \mathbb{H}^{N}} \frac{\left( \overline{u}\left( \delta_{S_s^{-\frac{1}{2s}}}(\xi)\right) -  \overline{u} \left( \delta_{S_s^{-\frac{1}{2s}}}(\eta)\right)\right)(v(\xi) - v(\eta) )}{|\eta^{-1} \circ \xi|^{Q+2s}} d\xi d\eta\\
 &= S_{s}^{\frac{Q-2s}{2s}}\iint\limits_{\mathbb{H}^{N} \ \mathbb{H}^{N}} \frac{\left( \overline{u}\left( \xi'\right) -  \overline{u} \left( \eta'\right)\right)(v(\delta_{S_s^{\frac{1}{2s}}}(\xi')) - v(\delta_{S_s^{\frac{1}{2s}}}(\eta')) )}{|\eta'^{-1} \circ \xi'|^{Q+2s}} d\xi' d\eta' \\
&= S_{s}^{\frac{Q}{2s}} \int\limits_{\mathbb{H}^N} |\overline{u}(\xi')|^{Q^*_s - 2}\overline{u}(\xi') v(\delta_{S_s^{\frac{1}{2s}}}(\xi')) d\xi' \\
&= \int\limits_{\mathbb{H}^N} |u^*(\xi)|^{Q^*_s - 2} u^*(\xi) v(\xi)  d\xi. \\
\end{align*}
Hence, $u^*$ is a weak solution to the problem \eqref{eqn3.2}. Moreover, on taking $v = u^*,$ we get
\begin{align*}
    \iint\limits_{\mathbb{H}^{N} \ \mathbb{H}^{N}} \frac{ |u^*(\xi) -  u^*(\eta)|^2}{|\eta^{-1} \circ \xi|^{Q+2s}} d\xi d\eta &= \int\limits_{\mathbb{H}^N} |u^*(\xi)|^{Q^*_s} d\xi \\
    &= \int\limits_{\mathbb{H}^N} |\overline{u}\left( \delta_{S_s^{-\frac{1}{2s}}}(\xi)\right)|^{Q^*_s} d\xi \\
    &= S_s^{\frac{Q}{2s}}.
\end{align*}
\end{proof}
Next, let $\varepsilon > 0$ be arbitrary. Define,
\begin{equation} \label{Ue}
    U_{\varepsilon}(\xi) = \frac{1}{\varepsilon^{\frac{Q-2s}{2}}} u^*\left( \delta_{\frac{1}{\varepsilon}}(\xi)\right).
\end{equation}
\begin{lemma}
The function $U_\varepsilon$ is a solution to Problem \eqref{eqn3.2}, and  it satisfies 
$$ \iint\limits_{\mathbb{H}^{N} \ \mathbb{H}^{N}} \frac{ |U_{\varepsilon}(\xi) -  U_{\varepsilon}(\eta)|^2}{|\eta^{-1} \circ \xi|^{Q+2s}} d\xi d\eta = \int\limits_{\mathbb{H}^N} |U_{\varepsilon}(\xi)|^{Q^*_s} d\xi 
 = S_{s}^{\frac{Q}{2s}}$$
for any $\varepsilon > 0$.
\end{lemma}
\begin{proof}
Let $\delta_{\frac{1}{\varepsilon}}(\xi)=\xi'.$ So $\varepsilon^{-Q}d \xi=d\xi'.$ Consider for any $v \in HW^{s,2}(\mathbb{H}^N)$,
\begin{align*}
\iint\limits_{\mathbb{H}^{N} \ \mathbb{H}^{N}} &\frac{ (U_{\varepsilon}(\xi) -  U_{\varepsilon}(\eta))(v(\xi) - v(\eta) )}{|\eta^{-1} \circ \xi|^{Q+2s}} d\xi d\eta \\
&=\frac{1}{\varepsilon^{\frac{Q-2s}{2}}} \iint\limits_{\mathbb{H}^{N} \ \mathbb{H}^{N}} \frac{\left( u^*\left( \delta_{\frac{1}{\varepsilon}}(\xi)\right) -  u^*\left( \delta_{\frac{1}{\varepsilon}}(\eta)\right)\right)(v(\xi) - v(\eta) )}{|\eta^{-1} \circ \xi|^{Q+2s}} d\xi d\eta \\
&= \varepsilon^{\frac{Q-2s}{2}} \iint\limits_{\mathbb{H}^{N} \ \mathbb{H}^{N}} \frac{\left( u^*( \xi' ) -  u^*(\eta')\right)(v(\delta_{\varepsilon}(\xi')) - v(\delta_{\varepsilon}(\eta')) )}{|\eta'^{-1} \circ \xi'|^{Q+2s}} d\xi' d\eta' \\
&= \varepsilon^{\frac{Q-2s}{2}} \int\limits_{\mathbb{H}^N} |u^*(\xi')|^{Q^*_s - 2} u^*(\xi')v(\delta_{\varepsilon}(\xi'))  d\xi' \\
&= \int\limits_{\mathbb{H}^N} |U_{\varepsilon}(\xi)|^{Q^*_s - 2} U_{\varepsilon}(\xi) v(\xi)  d\xi.
\end{align*}
On putting $v = U_{\varepsilon},$ we get 
\begin{align*}
    \iint\limits_{\mathbb{H}^{N} \ \mathbb{H}^{N}} \frac{ |U_{\varepsilon}(\xi) -  U_{\varepsilon}(\eta)|^2}{|\eta^{-1} \circ \xi|^{Q+2s}} d\xi d\eta &= \int\limits_{\mathbb{H}^N} |U_{\varepsilon}(\xi)|^{Q^*_s} d\xi \\
    &=  \int\limits_{\mathbb{H}^N} |u^*(\xi)|^{Q^*_s} d\xi \\
 &= S_{s}^{\frac{Q}{2s}}.
\end{align*}
\end{proof}
Define, 
\begin{equation} \label{ue}
    u_{\varepsilon}(\xi) = U_{\varepsilon}(\xi)\cdot\phi(\xi),
\end{equation}
where $U_\varepsilon$ is defined in \eqref{Ue}, and $\phi \in C^{\infty}{(\mathbb{H}^N)}$ is such that $0 \leq \phi \leq 1$ and
\begin{align} \label{eqn3.8}
    \begin{split}
        \phi &= \begin{cases}
1 \text{ in } B_r(0),\\
0 \text{ in } \mathbb{H}^N \setminus B_{2r}(0).
\end{cases} 
    \end{split}
\end{align}
for some  $r > 0$  such that $B_{4r}(0) \subset \Omega.$
\begin{lemma}\label{eqn3.7}
    Let $\rho > 0.$ If $|\xi| > \rho$, then
\begin{equation*}
    |u_{\varepsilon}(\xi)| \leq |U_{\varepsilon}(\xi)| \leq C \varepsilon^{\frac{Q-2s}{2}}.
\end{equation*}
for any $\varepsilon > 0$ and for some positive constant $C.$
\end{lemma}
\begin{proof}
    Let $\rho > 0$ and let $|\xi| > \rho.$ Then,
\begin{align*}
   |U_{\varepsilon}(\xi)| &=\frac{1}{\varepsilon^{\frac{Q-2s}{2}}} \frac{C}{\left( \frac{t^2}{{\varepsilon^4 S_s^{\frac{2}{s}}}} + \left(1 + \frac{|x|^2}{\varepsilon^2 S_s^{\frac{1}{s}}} + \frac{|y|^2}{\varepsilon^2 S_s^{\frac{1}{s}}}\right)^2  \right)^{\frac{Q-2s}{4}}} \\
   &\leq \frac{1}{\varepsilon^{\frac{Q-2s}{2}}} \frac{C}{\left( 1 + \frac{t^2}{{\varepsilon^4 S_s^{\frac{2}{s}}}} + \left( \frac{|x|^2}{\varepsilon^2 S_s^{\frac{1}{s}}} + \frac{|y|^2}{\varepsilon^2 S_s^{\frac{1}{s}}}\right)^2  \right)^{\frac{Q-2s}{4}}} \\
   &\leq \frac{1}{\varepsilon^{\frac{Q-2s}{2}}} \frac{C}{\left( 1 + |\delta_\frac{1}{\varepsilon  S_{s}^{\frac{1}{2s}}}(\xi)|^4  \right)^{\frac{Q-2s}{4}}} \\
   &\leq \frac{1}{\varepsilon^{\frac{Q-2s}{2}}} \frac{\varepsilon^{Q-2s}C}{\left( \varepsilon^4 + C\rho^4  \right)^{\frac{Q-2s}{4}}} \\
   &\leq C \varepsilon^{\frac{Q-2s}{2}}.
\end{align*}
\end{proof}

\begin{lemma}\label{lem4}
   Let $\rho > 0.$ If $|\xi| > \rho$, then
\begin{equation}\label{eqn3.9}
    |\nabla_{\mathbb{H}}u_\varepsilon(\xi)| \leq C \varepsilon^\frac{Q-2s}{2} 
\end{equation}
\end{lemma}
\begin{proof}
    We have,
\begin{align*}
    |\nabla_{\mathbb{H}} U_{\varepsilon}(\xi)| &= (Q-2s) S_s^{-\frac{1}{2s}} \varepsilon^\frac{-Q+2s -2}{2} \frac{\left(\frac{|x|^2}{\varepsilon^2 S_s^{\frac{1}{s}}} + \frac{|y|^2}{\varepsilon^2 S_s^{\frac{1}{s}}} \right)^{\frac{1}{2}}}{\left( \frac{t^2}{{\varepsilon^4 S_s^{\frac{2}{s}}}} + \left(1 + \frac{|x|^2}{\varepsilon^2 S_s^{\frac{1}{s}}} + \frac{|y|^2}{\varepsilon^2 S_s^{\frac{1}{s}}}\right)^2  \right)^{\frac{Q-2s+2}{4}}} \\
    &\leq C \varepsilon^\frac{-Q+2s-2}{2} \frac{|\delta_\frac{1}{\varepsilon S_{s}^{\frac{1}{2s}}}(\xi)|}{|\delta_\frac{1}{\varepsilon S_{s}^{\frac{1}{2s}}}(\xi)|^{Q-2s+2}} \\
    &\leq C \varepsilon^\frac{-Q+2s-2}{2} |\delta_\frac{1}{\varepsilon S_{s}^{\frac{1}{2s}}}(\xi)|^{-Q+2s-1} \\
    &\leq C\frac{1}{|\xi|^{Q-2s+1}} \varepsilon^\frac{-Q+2s-2}{2} \cdot \varepsilon^{Q-2s+1} \\
    &= C \frac{1}{|\xi|^{Q-2s+1}} \varepsilon^\frac{Q-2s}{2}.
\end{align*}
For $|\xi| > \rho,$
\begin{equation*}
     |\nabla_{\mathbb{H}} U_{\varepsilon}(\xi)| \leq C \varepsilon^\frac{Q-2s}{2},
\end{equation*}
 and hence it follows that for $|\xi|>\rho,$
\begin{equation*}
     |\nabla_{\mathbb{H}} u_{\varepsilon}(\xi)| \leq C \varepsilon^\frac{Q-2s}{2}.
\end{equation*}
\end{proof}

\begin{lemma}
    Let $r$ be as in \eqref{eqn3.8}. Then there exists a constant $r_b > 0$, depending on $r$, such that the following estimates hold for all $\varepsilon > 0$ and for some constant $C > 0$:

\begin{enumerate}
    \item If $\xi \in \mathbb{H}^N,\, \eta \in \mathbb{H}^N \setminus B_r(0)$ satisfy  $|\eta^{-1} \circ \xi| \leq r_b,$ then
    \begin{equation} \label{eqn3.11}
    |u_{\varepsilon}(\xi) -u_{\varepsilon}(\eta)|\leq C \varepsilon^\frac{Q-2s}{2}|\eta^{-1} \circ \xi|,
\end{equation}
    \item If $\xi,\,\eta \in  \mathbb{H}^N \setminus B_r(0),$ then
\begin{equation} \label{eqn3.12}
    |u_{\varepsilon}(\xi) -u_{\varepsilon}(\eta)|\leq C \varepsilon^\frac{Q-2s}{2}\min\{1,|\eta^{-1} \circ \xi|\}.
\end{equation}
\end{enumerate}

\end{lemma}
\begin{proof}
    By stratified Lagrange's Mean value theorem~\cite[Theorem 20.3.1]{bonfiglioli2007}, there exists positive constant $C$ and $b$ such that
\begin{equation}\label{l5e1}
    |u_{\varepsilon}(\xi) -u_{\varepsilon}(\eta)|\leq C |\eta^{-1} \circ \xi| \sup\limits_{|\zeta| \leq b  |\eta^{-1} \circ \xi|}|\nabla_{\mathbb{H}}u_\varepsilon(\eta \circ \zeta)|.
\end{equation}
If $b \leq 1,$ then for all $\xi \in \mathbb{H}^N, \,\eta \in \mathbb{H}^N \setminus B_r(0),$ and $|\eta^{-1} \circ \xi| \leq \frac{r}{2},$ it hold that
\begin{align*}
    |\eta \circ \zeta| &\geq |\eta| - |\zeta| \\
    &\geq r - b|\eta^{-1} \circ \xi| \\
    &\geq r - \frac{br}{2}\geq \frac{r}{2}.
\end{align*}
On choosing $\rho = \frac{r}{2},$ from Lemma \ref{lem4} and \eqref{l5e1}, we get 
\begin{equation*}
    |u_{\varepsilon}(\xi) -u_{\varepsilon}(\eta)|\leq C \varepsilon^\frac{Q-2s}{2}|\eta^{-1} \circ \xi|.
\end{equation*}
 If $b > 1,$ then for all  $\xi \in \mathbb{H}^N,\, \eta \in \mathbb{H}^N \setminus B_r(0),$ and  $|\eta^{-1} \circ \xi| \leq \frac{r}{2b},$ we have
\begin{align*}
    |\eta \circ \zeta| &\geq |\eta| - |\zeta| \\
    &\geq r - b|\eta^{-1} \circ \xi| \\
    &\geq r - \frac{br}{2b}\geq \frac{r}{2}.
\end{align*}
On choosing $\rho = \frac{r}{2},$ from Lemma \ref{lem4} and \eqref{l5e1}, we get
\begin{equation*}
    |u_{\varepsilon}(\xi) -u_{\varepsilon}(\eta)|\leq C \varepsilon^\frac{Q-2s}{2}|\eta^{-1} \circ \xi|.
\end{equation*}
Thus, \eqref{eqn3.11} holds with $ r_b=\begin{cases} \frac{r}{2b} & \text{if } b>1,\\  \frac{r}{2} & \text{if }b \leq 1.\end{cases} $

\noindent Now to prove \eqref{eqn3.12}, let $\xi, \eta \in  \mathbb{H}^N \setminus B_r(0).$ Then  \\
If $|\eta^{-1} \circ \xi| \leq r_b,$ we have,
\begin{equation*}
    |u_{\varepsilon}(\xi) -u_{\varepsilon}(\eta)|\leq C \varepsilon^\frac{Q-2s}{2}|\eta^{-1} \circ \xi|.
\end{equation*}
If $|\eta^{-1} \circ \xi| > r_b,$ we have,
\begin{equation*}
    |u_{\varepsilon}(\xi) -u_{\varepsilon}(\eta)|\leq |u_{\varepsilon}(\xi)| + |u_{\varepsilon}(\eta)| \leq C \varepsilon^\frac{Q-2s}{2}.
\end{equation*}
Therefore, for $\xi ,\eta \in  \mathbb{H}^N \setminus B_r(0),$
\begin{equation*}
    |u_{\varepsilon}(\xi) -u_{\varepsilon}(\eta)|\leq C \varepsilon^\frac{Q-2s}{2}\min\{1,|\eta^{-1} \circ \xi|\}.
\end{equation*}
\end{proof}

\begin{lemma} \label{lem6}
Let $s \in (0, 1)$. Then the following estimate holds:
\begin{equation*}
    \iint\limits_{\mathbb{H}^{N} \ \mathbb{H}^{N}}  \frac{ |u_{\varepsilon}(\xi) -  u_{\varepsilon}(\eta)|^2}{|\eta^{-1} \circ \xi|^{Q+2s}} d\xi d\eta \leq S_s^{\frac{Q}{2s}} + O(\varepsilon^{Q-2s})
\end{equation*}
as $\varepsilon \to 0$.
\end{lemma}
\begin{proof}
Let \begin{align*}
    \begin{split}
        D&=\{(\xi, \eta) \in \mathbb{H}^{N} \times \mathbb{H}^{N} : \xi \in B_r(0), \eta \in \mathbb{H}^N \setminus B_r(0), |\eta^{-1} \circ \xi| > r_b \} \text{ and } \\
E &= \{(\xi, \eta) \in \mathbb{H}^{N} \times \mathbb{H}^{N} : \xi \in B_r(0), \eta \in \mathbb{H}^N \setminus B_r(0), |\eta^{-1} \circ \xi| \leq r_b \}.
    \end{split}
\end{align*}
Then
\begin{align*}
\begin{split}
     &\iint\limits_{\mathbb{H}^{N} \ \mathbb{H}^{N}}  \frac{ |u_{\varepsilon}(\xi) -  u_{\varepsilon}(\eta)|^2}{|\eta^{-1} \circ \xi|^{Q+2s}} d\xi d\eta =  \int\limits_{B_r} \int\limits_{B_r} \frac{ |U_{\varepsilon}(\xi) -  U_{\varepsilon}(\eta)|^2}{|\eta^{-1} \circ \xi|^{Q+2s}} d\xi d \eta \\& + 2 \iint\limits_{D} \frac{ |u_{\varepsilon}(\xi) -  u_{\varepsilon}(\eta)|^2}{|\eta^{-1} \circ \xi|^{Q+2s}} d\xi d\eta  
      + 2 \iint\limits_{E} \frac{ |u_{\varepsilon}(\xi) -  u_{\varepsilon}(\eta)|^2}{|\eta^{-1} \circ \xi|^{Q+2s}} d\xi d\eta \\& +  \int\limits_{B^c_r(0)} \int\limits_{B^c_r(0)} \frac{ |u_{\varepsilon}(\xi) -  u_{\varepsilon}(\eta)|^2}{|\eta^{-1} \circ \xi|^{Q+2s}} d\xi d\eta.
\end{split} 
\end{align*}
Now, we see that
\begin{align*}
    \int\limits_{B^c_r(0)} \int\limits_{B^c_r(0)} \frac{ |u_{\varepsilon}(\xi) -  u_{\varepsilon}(\eta)|^2}{|\eta^{-1} \circ \xi|^{Q+2s}} d\xi d\eta &\leq C \varepsilon^{Q-2s}\int\limits_{B^c_r(0)} \int\limits_{B^c_r(0)}  \frac{  \min\{1,|\eta^{-1} \circ \xi|^2\}}{|\eta^{-1} \circ \xi|^{Q+2s}} d\eta d\xi \\
    &\leq C \varepsilon^{Q-2s}\int\limits_{ B_{2r}(0)} \int\limits_{\mathbb{H}^{N}}  \frac{  \min\{1,|\eta^{-1} \circ \xi|^2\}}{|\eta^{-1} \circ \xi|^{Q+2s}} d\eta d\xi\\
    &= O(\varepsilon^{Q-2s}).
\end{align*}

\begin{align*}
    \iint\limits_{E} \frac{ |u_{\varepsilon}(\xi) -  u_{\varepsilon}(\eta)|^2}{|\eta^{-1} \circ \xi|^{Q+2s}}d\eta d\xi  &\leq  \int\limits_{B_r(0)} \int\limits_{B_r^c \cap \{|\eta^{-1} \circ \xi| \leq r_b\}} \frac{ |u_{\varepsilon}(\xi) -  u_{\varepsilon}(\eta)|^2}{|\eta^{-1} \circ \xi|^{Q+2s}} d\eta d\xi \\
    &\leq C \varepsilon^{Q-2s} \int\limits_{B_r(0)} \int\limits_{B_r^c \cap \{|\eta^{-1} \circ \xi| \leq r_b\}} \frac{ |\eta^{-1} \circ \xi|^2}{|\eta^{-1} \circ \xi|^{Q+2s}} d\eta d\xi \\
    &\leq C \varepsilon^{Q-2s} \int\limits_{B_r(0)} \int\limits_{B_r^c \cap \{|\eta^{-1} \circ \xi| \leq r_b\}} \frac{ 1}{|\eta^{-1} \circ \xi|^{Q+2s-2}} d\eta d\xi \\
    &\leq C \varepsilon^{Q-2s} \int\limits_{B_r(0)} \int\limits_{|\zeta
    |\leq r_b} \frac{ 1}{|\zeta|^{Q+2s-2}} d\zeta d\xi \\
    &= O(\varepsilon^{Q-2s}).
\end{align*}
Since \( u_\varepsilon(\xi) = U_\varepsilon(\xi) \) for all \( \xi \in B_r(0) \), it follows that for any \( (\xi, \eta) \in D \),
\begin{align*}
    |u_\varepsilon(\xi) - u_\varepsilon(\eta)|^2 &= |U_\varepsilon(\xi) - u_\varepsilon(\eta)|^2 \\
    &= |(U_\varepsilon(\xi) - U_\varepsilon(\eta)) + (U_\varepsilon(\eta) - u_\varepsilon(\eta))|^2 \\
    &\leq |U_\varepsilon(\xi) - U_\varepsilon(\eta)|^2 + |U_\varepsilon(\eta) - u_\varepsilon(\eta)|^2 \\
    &\quad+ 2|U_\varepsilon(\xi) - U_\varepsilon(\eta)| |U_\varepsilon(\eta) - u_\varepsilon(\eta)|.
\end{align*}
Therefore,
\begin{align*}
     &\iint\limits_{D} \frac{ |u_{\varepsilon}(\xi) -  u_{\varepsilon}(\eta)|^2}{|\eta^{-1} \circ \xi|^{Q+2s}}d\eta d\xi \leq \iint\limits_{D} \frac{|U_\varepsilon(\xi) - U_\varepsilon(\eta)|^2 }{|\eta^{-1} \circ \xi|^{Q+2s}}d\eta d\xi \\& + \iint\limits_{D} \frac{ |U_\varepsilon(\eta) - u_\varepsilon(\eta)|^2}{|\eta^{-1} \circ \xi|^{Q+2s}}d\eta d\xi + 2\iint\limits_{D} \frac{ ||U_\varepsilon(\xi) - U_\varepsilon(\eta)| |U_\varepsilon(\eta) - u_\varepsilon(\eta)|}{|\eta^{-1} \circ \xi|^{Q+2s}}d\eta d\xi.
\end{align*}
Consider,
\begin{align*}
     \iint\limits_{D} \frac{ |U_\varepsilon(\eta) - u_\varepsilon(\eta)|^2}{|\eta^{-1} \circ \xi|^{Q+2s}}d\eta d\xi &\leq 4\iint\limits_{D} \frac{ |U_\varepsilon(\eta)|^2}{|\eta^{-1} \circ \xi|^{Q+2s}}d\eta d\xi \\
     &\leq C \varepsilon^{Q-2s} \int\limits_{B_r(0)} \int\limits_{B_r^c \cap \{|\eta^{-1} \circ \xi| > r_b\}} \frac{1}{|\eta^{-1} \circ \xi|^{Q+2s}} d\eta d\xi \\
     &\leq C \varepsilon^{Q-2s} \int\limits_{B_r(0)} \int\limits_{B_r^c} \frac{1}{|\zeta|^{Q+2s}} d\zeta d\xi \\
      &= O( \varepsilon^{Q-2s}).
\end{align*}
Now for $(\xi,  \eta) \in D,$ we have,
\begin{align*}
    |U_\varepsilon(\xi) ||U_\varepsilon(\eta)| &\leq C \varepsilon^\frac{Q-2s}{2} \frac{C\varepsilon^\frac{-(Q-2s)}{2}}{\left( \frac{t^2}{{\varepsilon^4 S_s^{\frac{2}{s}}}} + \left(1 + \frac{|x|^2}{\varepsilon^2 S_s^{\frac{1}{s}}} + \frac{|y|^2}{\varepsilon^2 S_s^{\frac{1}{s}}}\right)^2  \right)^{\frac{Q-2s}{4}}} \\
    &\leq \frac{C}{\left( \frac{t^2}{{\varepsilon^4 S_s^{\frac{2}{s}}}} + \left(1 + \frac{|x|^2}{\varepsilon^2 S_s^{\frac{1}{s}}} + \frac{|y|^2}{\varepsilon^2 S_s^{\frac{1}{s}}}\right)^2  \right)^{\frac{Q-2s}{4}}}.
\end{align*}
Let \( \xi' = \delta_{\frac{1}{\varepsilon S_s^{\frac{1}{2s}}}}(\xi) \). Then the change of variables yields \( d\xi = \varepsilon^Q S_s^{\frac{Q}{2s}}\, d\xi' \), and the condition \( |\xi| = r \) implies \( |\xi'| = \frac{r}{\varepsilon S_s^{\frac{1}{2s}}} = r_\varepsilon \). Therefore, applying this transformation along with polar co-ordinates on \( \mathbb{H}^N \) (see \cite[Proposition 1.15]{folland1982hardy}), we obtain,
\begin{align*}
    \iint\limits_D \frac{|U_\varepsilon(\xi)| |U_\varepsilon(\eta)|}{|\eta^{-1} \circ \xi|^{Q+2s}} d\xi  d\eta &\leq C \iint\limits_D \frac{\left( \frac{t^2}{{\varepsilon^4 S_s^{\frac{2}{s}}}} + \left(1 + \frac{|x|^2}{\varepsilon^2 S_s^{\frac{1}{s}}} + \frac{|y|^2}{\varepsilon^2 S_s^{\frac{1}{s}}}\right)^2  \right)^{-\frac{Q-2s}{4}}}{|\eta^{-1} \circ \xi|^{Q+2s}} d\xi d\eta  \\
    &\leq  C \iint\limits_D \frac{1}{\left( 1 + |\delta_\frac{1}{\varepsilon S_{s}^{\frac{1}{2s}}}(\xi)|^4  \right)^{\frac{Q-2s}{4}}} \frac{1}{|\eta^{-1} \circ \xi|^{Q+2s}}  d\xi d\eta \\
    &\leq  C \int\limits_{B_r(0)} \frac{1}{\left( 1 + |\delta_\frac{1}{\varepsilon S_{s}^{\frac{1}{2s}}}(\xi)|^4  \right)^{\frac{Q-2s}{4}}} d \xi \\
    &\leq C \varepsilon^Q \int\limits_{B_{r_\varepsilon}} \frac{1}{\left( 1 + |\xi'|^4  \right)^{\frac{Q-2s}{4}}} d \xi' \\
    &\leq C \varepsilon^Q \left( 1 + \int\limits_{B_{r_\varepsilon} \setminus B_1} \frac{1}{\left( 1 + |\xi'|^4  \right)^{\frac{Q-2s}{4}}} d \xi' \right) \\
    &\leq C \varepsilon^Q \left( 1 + \int\limits_{B_{r_\varepsilon} \setminus B_1} \frac{1}{|\xi'|^{Q-2s}} d \xi' \right) \\
     &\leq C \varepsilon^Q \left( 1 + \int\limits_1^{r_\varepsilon} \frac{1}{\omega^{Q-2s}} w^{Q-1} d \omega \right) \\
     &\leq C \varepsilon^Q \left( 1 + \int\limits_1^{r_\varepsilon} {\omega^{-Q+2s+Q-1}} d \omega \right) \\
     &\leq C \varepsilon^Q \left( 1 + \left(\frac{r}{\varepsilon S_{s}^{\frac{1}{2s}}}\right)^{2s}\right) \\
     &\leq C \varepsilon^{Q - 2s}.
\end{align*}
Moreover, 
\begin{align*}
    \iint\limits_D \frac{|U_\varepsilon(\xi)| |U_\varepsilon(\eta) - u_\varepsilon(\eta)|}{|\eta^{-1} \circ \xi|^{Q+2s}} d\xi  d\eta  &\leq \iint\limits_D \frac{|U_\varepsilon(\xi)| (|U_\varepsilon(\eta)| + |u_\varepsilon(\eta)|)}{|\eta^{-1} \circ \xi|^{Q+2s}} d\xi  d\eta \\
    &\leq 2 \iint\limits_D \frac{|U_\varepsilon(\xi)| |U_\varepsilon(\eta)|}{|\eta^{-1} \circ \xi|^{Q+2s}} d\xi  d\eta \\
    &= O(\varepsilon^{Q - 2s}).
\end{align*}
Further,
\begin{align*}
    \iint\limits_D \frac{|U_\varepsilon(\eta)| |U_\varepsilon(\eta) - u_\varepsilon(\eta)|}{|\eta^{-1} \circ \xi|^{Q+2s}} d\xi  d\eta  &\leq 2   \iint\limits_D \frac{|U_\varepsilon(\eta)|^2 }{|\eta^{-1} \circ \xi|^{Q+2s}} d\xi  d\eta \\
    &\leq C  \iint\limits_D \frac{\varepsilon^{Q - 2s}}{|\eta^{-1} \circ \xi|^{Q+2s}} d\xi  d\eta \\
    &\leq C \varepsilon^{Q - 2s}. 
\end{align*}
Therefore,
\begin{align*}
    \iint\limits_{\mathbb{H}^{N} \ \mathbb{H}^{N}}  \frac{ |u_{\varepsilon}(\xi) -  u_{\varepsilon}(\eta)|^2}{|\eta^{-1} \circ \xi|^{Q+2s}} d\xi d\eta &=  \int\limits_{B_r(0)} \int\limits_{B_r(0)} \frac{ |U_{\varepsilon}(\xi) -  U_{\varepsilon}(\eta)|^2}{|\eta^{-1} \circ \xi|^{Q+2s}} d\xi d \eta \\
    &\quad + 2  \iint\limits_{D} \frac{ |U_\varepsilon(\eta) - u_\varepsilon(\eta)|^2}{|\eta^{-1} \circ \xi|^{Q+2s}}d\eta d\xi + O(\varepsilon^{Q-2s}) \\
    &\leq \iint\limits_{\mathbb{H}^{N} \ \mathbb{H}^{N}}  \frac{ |U_{\varepsilon}(\xi) -  U_{\varepsilon}(\eta)|^2}{|\eta^{-1} \circ \xi|^{Q+2s}} d\xi d\eta + O(\varepsilon^{Q-2s}) \\
    &\leq S_s^{\frac{Q}{2s}} + O(\varepsilon^{Q-2s}).
\end{align*}
\end{proof}

\begin{lemma} \label{lem7}
    Let $s \in (0,1).$ Then, 
    \begin{equation*}
        \int\limits_{\mathbb{H}^N} |u_\varepsilon(\xi)|^2 d \xi \geq  O(\varepsilon^{Q-4s}) + C\varepsilon^{2s}
    \end{equation*}
    and
    \begin{equation*}
        \int\limits_{\mathbb{H}^N} |u_\varepsilon(\xi)|^{Q^*_s} d\xi =  S_s^{\frac{Q}{2s}} + O( \varepsilon^Q).
    \end{equation*}
\end{lemma}
\begin{proof}
Define \( \xi' = \delta_{\frac{1}{\varepsilon}}(\xi) \). Then, under this change of variables, we have \( d\xi = \varepsilon^Q\, d\xi' \), and the condition \( |\xi| = r \) implies \( |\xi'| = \frac{r}{\varepsilon} \). Using this transformation, we obtain:
    \begin{align*}
    \int\limits_{\mathbb{H}^N} |u_\varepsilon(\xi)|^2  d \xi &= \int\limits_{B_r(0)} |U_\varepsilon(\xi)|^2 d \xi + \int\limits_{B_{2r} \setminus B_r(0)} |(\phi(\xi)U_\varepsilon(\xi))|^2 d \xi \\
    &\geq  \varepsilon^{-(Q-2s)} \int\limits_{B_r} |u^*( \delta_{\frac{1}{\varepsilon}}(\xi))|^2 d \xi\\
    &\geq  \varepsilon^{2s} \int\limits_{B_{\frac{r}{\varepsilon}}} |u^*(\xi'))|^2 d \xi', \\
    &\geq  \varepsilon^{2s} \int\limits_{B_{\frac{r}{\varepsilon}} \setminus B_R} |u^*(\xi'))|^2 d \xi' 
\end{align*}
for any $0 < R < \frac{r}{\varepsilon}.$ By using polar co-ordinates on $\mathbb{H}^N$(see \cite[Proposition 1.15]{folland1982hardy}), we get
\begin{align*}
    \int\limits_{\mathbb{H}^N} |u_\varepsilon(\xi)|^2  d \xi &\geq \varepsilon^{2s} \int\limits_{B_{\frac{r}{\varepsilon}} \setminus B_R} \frac{C}{\left( \frac{t^2}{{ S_s^{\frac{2}{s}}}} + \left(1 + \frac{|x|^2}{ S_s^{\frac{1}{s}}} + \frac{|y|^2}{ S_s^{\frac{1}{s}}}\right)^2  \right)^{\frac{Q-2s}{2}}} \\
    &\eqsim C \varepsilon^{2s} \int\limits_R^{\frac{r}{\varepsilon}} \frac{\omega^{Q-1}}{\omega^{2Q-4s}} d \omega \\
    &= C \varepsilon^{2s} \int\limits_R^{\frac{r}{\varepsilon}} \omega^{-Q+4s-1} d \omega \\
    &= C \varepsilon^{2s} \left(\varepsilon^{Q-4s} - R^{-Q+4s} \right)\\
    &= O(\varepsilon^{Q-4s}) + C\varepsilon^{2s}.
\end{align*}
Moreover,
\begin{align*}
  \int\limits_{\mathbb{H}^N} |u_\varepsilon(\xi)|^{Q^*_s}   d\xi &= \int\limits_{\mathbb{H}^N} |U_\varepsilon(\xi)|^{Q^*_s}   d\xi + \int\limits_{\mathbb{H}^N} \left( | \phi(\xi) |^{Q^*_s} - 1 \right) |U_\varepsilon(\xi)|^{Q^*_s} d\xi \\
  &= S_s^{\frac{Q}{2s}}  + \int\limits_{B^c_r(0)} \left( | \phi(\xi) |^{Q^*_s} - 1 \right) |U_\varepsilon(\xi)|^{Q^*_s}   d\xi \\
  &= S_s^{\frac{Q}{2s}}  +  \varepsilon^{-Q}\int\limits_{B^c_r(0)} \left( | \phi(\xi) |^{Q^*_s} - 1 \right) | u^*(\delta_{\frac{1}{\varepsilon}}(\xi)) |^{Q^*_s}   d\xi \\
  &= S_s^{\frac{Q}{2s}}  +  \varepsilon^{-Q}\int\limits_{B^c_r(0)} \frac{C}{\left( \frac{t^2}{{\varepsilon^4 S_s^{\frac{2}{s}}}} + \left(1 + \frac{|x|^2}{\varepsilon^2 S_s^{\frac{1}{s}}} + \frac{|y|^2}{\varepsilon^2 S_s^{\frac{1}{s}}}\right)^2  \right)^{\frac{Q}{2}}}  d\xi \\
  &\eqsim S_s^{\frac{Q}{2s}}  + C  \varepsilon^{Q} \int\limits_{B^c_r(0)} |\xi|^{-2Q}   d\xi \\
  &= S_s^{\frac{Q}{2s}} + O( \varepsilon^Q).
\end{align*}
\end{proof}

\begin{lemma}\label{lem3.9}
Let $\lambda>0.$ Then $S_{s,\lambda}<S_s.$ Moreover, the infimum 
$$S_{s, \lambda} = \inf\limits_{\substack{u \in X_0 \\ \|u\|_{L^{Q^*_s}(\Omega)}= 1} }\  [u]^2 - \lambda\|u\|_{L^2(\mathbb{H}^N)}^2$$
is achieved.
\end{lemma}
\begin{proof}
Let $$S_{s, \lambda}(u)=\frac{\iint\limits_{\mathcal{S}} \frac{|u(\xi) - u(\eta)|^2}{|\eta^{-1} \circ \xi|^{Q+2s}} d\xi d\eta - \lambda \int\limits_{\mathbb{H}^N} |u(\xi)|^2  d\xi }{\left( \int\limits_{\mathbb{H}^N} |u(\xi)|^{Q^*_s} d \xi \right)^{2/Q^*_s}}.$$
On applying Lemma \ref{lem6} and Lemma \ref{lem7},
    \begin{align*}
 S_{s, \lambda}(u_\varepsilon) &\leq \frac{S_s^{\frac{Q}{2s}} + O(\varepsilon^{Q-2s}) - \lambda C_s \varepsilon^{2s} }{\left(S_s^{\frac{Q}{2s}} + O(\varepsilon^{Q})\right)^{\frac{2}{Q^*_s}}} \\
&\leq S_s + \varepsilon^{2s} (O(\varepsilon^{Q-4s}) - \lambda C_s) \\ 
&< S_s.
\end{align*}
    Let $\{u_n\}$ be a minimizing sequence i.e. $\|u_n\|_{L^{Q^*_s}(\Omega)} = 1$ and $S_{s, \lambda} = [u_n]^2 - \lambda\|u_n\|^2_{L^2(\Omega)}.$ Since $\{u_n\}$ is bounded in $X_0,$
    \begin{align*}
        u_n  &\rightharpoonup u \text{ in } X_0, \\
        u_n  &\to u \text{ in } L^2(\Omega), \\
        u_n(\xi)  &\to u(\xi) \text{ a.e. on } \Omega.
    \end{align*}
    Set $v_n = u_n - u,$ so that $v_n \rightharpoonup 0$ in $X_0.$
    By definition, we have $S_s \leq [u_n]^2.$ This implies that $S_s - \lambda\|u\|_{L^2(\Omega)}^2 \leq S_{s, \lambda} + o(1).$ Consequently, $0 < S_s - S_{s, \lambda} \leq \lambda\|u\|_{L^2(\Omega)}^2.$

\noindent By Brezis-Lieb lemma,
\begin{align*}
\|u+v_{n}\|_{L^{Q^*_s}(\Omega)}^{Q_{s}^{*}} &= \|u\|_{L^{Q^*_s}(\Omega)}^{Q_{s}^{*}}+\|v_n\|_{L^{Q^*_s}(\Omega)}^{Q_{s}^{*}}+o(1)\\
1&=\| u\|_{L^{Q^*_s}(\Omega)}^{Q_{s}^{*}}+\| v_n\|_{L^{Q^*_s}(\Omega)}^{Q_{s}^{*}}+o(1) \\
1 &\leq\|u\|_{L^{Q^*_s}(\Omega)}^{2}+\|v_{n}\|_{L^{Q^*_s}(\Omega)}^{2}+o(1) \\
1 &\leq\|u\|_{L^{Q^*_s}(\Omega)}^{2}+\frac{1}{S_s}[v_{n}]^{2}+o(1).
\end{align*}
Again by Brezis-Lieb lemma,
\[\left[u_n\right]^{2}-\left[v_n\right]^{2}=[u]^{2}+o(1).\]
\begin{align*}
S_{s,\lambda}&=[u]^{2}+\left[v_n\right]^{2}-\lambda \| u \|_{L^2(\Omega)}^{2}+o(1) \\
1 &\leq\|u\|_{L^{Q^*_s}(\Omega)}^{2}+\frac{1}{S_s}\left[v_n\right]^{2}+o(1) \\
S_{s,\lambda} &\leq S_{s,\lambda}\|u\|_{L^{Q^*_s}(\Omega)}^{2}+\frac{S_{s,\lambda}}{S_s}\left[v_n\right]^{2}+o(1) \\
S_{s,\lambda} &\leq S_{s,\lambda}\|u\|_{L^{Q^*_s}(\Omega)}^{2}+\left[v_n\right]^{2}+o(1) \\
S_{s,\lambda} &\leq S_{s,\lambda}\|u\|_{L^{Q^*_s}(\Omega)}^{2}+S_{s,\lambda}-[u]^{2}+\lambda\|u\|_{L^2(\Omega)}^{2}+ o(1).\\
\end{align*}
Hence,
$$
\frac{[u]^{2}-\lambda\|u\|_{L^2(\Omega)}^{2}}{\|u\|_{L^{Q^*_s}(\Omega)}^{2}} \leq S_{s,\lambda}.
$$
This shows that the infimum is attained.
\end{proof}

We now proceed to prove the main result of this article:

\begin{proof}[Proof of Theorem \ref{thm1}]
    Let $u \in X_0$ be as in Lemma \ref{lem3.9}, that is,
$\|u\|_{L^{Q^*_s}(\Omega)} = 1$ and  $[u]^{2}-\lambda\|u\|_{L^2(\Omega)}^{2}= S_{s,\lambda}.$ We may as well assume that $u \geq 0$ (otherwise we replace $u$ by $|u|$). Since
$u$ is a minimizer, we obtain a Lagrange multiplier $\mu \in \mathbb{R}$ such that
\[(-\Delta_\mathbb{H})^su - \lambda u = \mu |u|^{Q^*_s -2}u . \]
In fact, $\mu = S_{s, \lambda}$ and $S_{s, \lambda} > 0$ since $\lambda < \lambda_1$. It follows that $S_{s,\lambda}^{\frac{1}{{Q_{s}^{*}}-2}}u$ satisfies the equation \eqref{eqn1.1}. Note that $u > 0$ on $\mathbb{H}^N,$ by the strong maximum principle\cite[Theorem 6]{ghosh2024compact}.
\end{proof}


\begin{thebibliography}{99}

\bibitem{mallick2018hardy}
Adimurthi, Mallick A.:
\newblock A {H}ardy type inequality on fractional order {S}obolev spaces on the {H}eisenberg group.
\newblock { Ann. Sc. Norm. Super. Pisa Cl. Sci. (5)} \textbf{18}(3), 917--949 (2018).  \url{https://doi.org/10.2422/2036-2145.201604_010}

\bibitem{alves2024brezis}
Alves C. O., Gandal S., Loiudice A., Tyagi J.:
\newblock A {B}r\'ezis-{N}irenberg type problem for a class of degenerate elliptic problems involving the {G}rushin operator.
\newblock { J. Geom. Anal.} \textbf{34}, 52 (2024). \url{https://www.doi.org/10.1007/s12220-023-01507-3}

\bibitem{bandle2002brezis}
Bandle, C., Benguria, R.:
\newblock The {B}r\'ezis-{N}irenberg problem on {$\mathbb S^3$}.
\newblock { J. Differential Equations} \textbf{178}(1), 264--279 (2002). \url{https://doi.org/10.1006/jdeq.2001.4006}

\bibitem{barrios2012some}
Barrios, B., Colorado, E., de Pablo, A., S\'anchez, U.:
\newblock On some critical problems for the fractional {L}aplacian operator.
\newblock { J. Differential Equations} \textbf{252}(11), 6133--6162 (2012). \url{https://doi.org/10.1016/j.jde.2012.02.023}

\bibitem{barrios2015critical}
Barrios, B., Colorado, E., Servadei, R., Soria, F.:
\newblock A critical fractional equation with concave-convex power nonlinearities.
\newblock { Ann. Inst. H. Poincar\'e{} C Anal. Non Lin\'eaire} \textbf{32}(4), 875--900 (2015). \url{https://doi.org/10.1016/j.anihpc.2014.04.003}

\bibitem{bonder2018concentration}
Bonder, J. F., Saintier, N., Silva, A.:
\newblock The concentration-compactness principle for fractional order {S}obolev spaces in unbounded domains and applications to the generalized fractional {B}rezis-{N}irenberg problem.
\newblock {  NoDEA Nonlinear Differential Equations Appl.} \textbf{25}, 52 (2018). \url{https://doi.org/10.1007/s00030-018-0543-5}

\bibitem{bonfiglioli2007}
Bonfiglioli, A., Lanconelli, E., Uguzzoni, F.:
\newblock {  Stratified {L}ie Groups And Potential Theory For Their Sub-{L}aplacians}.
\newblock Springer Monographs in Mathematics. Springer, Berlin (2007). \url{https://doi.org/10.1007/978-3-540-71897-0}

\bibitem{branson2013moser}
Branson, T. P., Fontana, L., Morpurgo, C.:
\newblock Moser-{T}rudinger and {B}eckner-{O}nofri's inequalities on the {CR} sphere.
\newblock {  Ann. of Math. (2)} \textbf{177}(1), 1--52 (2013). \url{http://dx.doi.org/10.4007/annals.2013.177.1.1}

\bibitem{brasco2018optimal}
Brasco, L., Squassina, M.:
\newblock Optimal solvability for a nonlocal problem at critical growth.
\newblock {  J. Differential Equations} \textbf{264}(3), 2242--2269 (2018). \url{https://doi.org/10.1016/j.jde.2017.10.019}

\bibitem{brezis1983positive}
Br\'ezis, H., Nirenberg, L.:
\newblock Positive solutions of nonlinear elliptic equations involving critical {S}obolev exponents.
\newblock {  Comm. Pure Appl. Math.} \textbf{36}(4), 437--477 (1983). \url{https://doi.org/10.1002/cpa.3160360405}

\bibitem{cantizano2019three}
Cantizano, N. A., Silva, A.:
\newblock Three solutions for a nonlocal problem with critical growth.
\newblock {  J. Math. Anal. Appl.} \textbf{469}(2), 841--851 (2019). \url{https://doi.org/10.1016/j.jmaa.2018.09.038}

\bibitem{carriao2019brezis}
Carri\~ao, P.~C., Lehrer, R., Miyagaki, O.~H., Vicente, A.
\newblock A {B}rezis-{N}irenberg problem on hyperbolic spaces.
\newblock {  Electron. J. Differential Equations} \textbf{2019}(67), 1--15 (2019)

\bibitem{citti1995semilinear}
Citti, G.:
\newblock Semilinear {D}irichlet problem involving critical exponent for the {K}ohn {L}aplacian.
\newblock {  Ann. Mat. Pura Appl. (4)} \textbf{169}(1), 375--392 (1995). \url{https://doi.org/10.1007/BF01759361}

\bibitem{cohn2001best}
Cohn, W. S., Lu, G.:
\newblock Best constants for {M}oser–{T}rudinger inequalities on the {H}eisenberg group.
\newblock {  Indiana Univ. Math. J.} \textbf{50}(4), 1567--1591, (2001). \url{https://doi.org/10.1512/iumj.2001.50.2138}

\bibitem{folland1982hardy}
Folland, G. B., Stein, E. M.:
\newblock Hardy Spaces On Homogeneous Groups.(MN-28)
\newblock Princeton University Press (1982)

\bibitem{frank2015extension}
Frank, R. L., Gonz\'alez, M. d. M., Monticelli, D. D., Tan, J.:
\newblock An extension problem for the {CR} fractional {L}aplacian.
\newblock {  Adv. Math.} \textbf{270}, 97--137 (2015). \url{https://doi.org/10.1016/j.aim.2014.09.026}

\bibitem{frank2012sharp}
Frank, R.~L., Lieb, E.~H.:
\newblock Sharp constants in several inequalities on the {H}eisenberg group.
\newblock {  Ann. of Math. (2)} \textbf{176}(1), 349--381 (2012). \url{https://dx.doi.org/10.4007/annals.2012.176.1.6}

\bibitem{garofalo2024optimal}
Garofalo, N., Loiudice, A., Vassilev, D.:
\newblock Optimal decay for solutions of nonlocal semilinear equations with critical exponent in homogeneous groups.
\newblock {  Proc. R. Soc. Edinb. A: Math.} 1--29 (2024). \url{https://doi.org/10.1017/prm.2024.58}

\bibitem{ghosh2024compact}
Ghosh, S., Kumar, V., Ruzhansky, M.:
\newblock Compact embeddings, eigenvalue problems, and subelliptic {B}rezis-{N}irenberg equations involving singularity on stratified {L}ie groups.
\newblock {  Math. Ann.} \textbf{388}(4), 4201--4249 (2024). \url{https://doi.org//10.1007/s00208-023-02609-7}

\bibitem{ghosh2024subelliptic}
Ghosh, S., Kumar, V., Ruzhansky, M.:
\newblock Subelliptic Nonlocal {B}rezis-{N}irenberg Problems on Stratified {L}ie Groups.
\newblock {  Commun. Contemp. Math.} 2550079, (2025). \url{https://doi.org/10.1142/S0219199725500798}

\bibitem{goel2020existence}
Goel, D., Sreenadh, K.:
\newblock Existence and nonexistence results for {K}ohn {L}aplacian with {H}ardy-{L}ittlewood-{S}obolev critical exponents.
\newblock {  J. Math. Anal. Appl.} \textbf{486}(2), 123915,  (2020). \url{https://doi.org/10.1016/j.jmaa.2020.123915}

\bibitem{han2020integral}
Han, Y.:
\newblock An integral type {B}rezis–{N}irenberg problem on the {H}eisenberg group.
\newblock {  J. Differential Equations} \textbf{269}(5), 4544--4565 (2020). \url{https://doi.org/10.1016/j.jde.2020.03.032}

\bibitem{loiudice2007semilinear}
Loiudice, A.:
\newblock Semilinear subelliptic problems with critical growth on {C}arnot groups.
\newblock {  Manuscripta Math.} \textbf{124}(2), 247--259 (2007). \url{https://doi.org/10.1007/s00229-007-0119-x}

\bibitem{loiudice2015critical}
Loiudice, A.:
\newblock Critical growth problems with singular nonlinearities on {C}arnot groups.
\newblock {  Nonlinear Anal.} \textbf{126}, 415--436 (2015). \url{https://doi.org/10.1016/j.na.2015.06.010}

\bibitem{loiudice2022existence}
Loiudice, A.:
\newblock{Existence results for critical problems involving $p$-sub-{L}aplacians on {C}arnot groups.}
\newblock{In: Ruzhansky, M., Wirth, J. (eds) Harmonic Analysis and Partial Differential Equations. Trends in Mathematics.} Birkhäuser, Cham (2022).  \url{https://doi.org/10.1007/978-3-031-24311-0_8}

\bibitem{manfredini2023holder}
Manfredini, M., Palatucci, G., Piccinini, M., Polidoro, S.:
\newblock H\"older continuity and boundedness estimates for nonlinear fractional equations in the {H}eisenberg group.
\newblock {  J. Geom. Anal.} \textbf{33}(3), 77 (2023). \url{https://doi.org/10.1007/s12220-022-01124-6}

\bibitem{mawhin2017brezis}
Mawhin, J., Molica Bisci, G.:
\newblock A {B}rezis-{N}irenberg type result for a nonlocal fractional operator.
\newblock {  J. Lond. Math. Soc. (2)}  \textbf{95}(1), 73--93 (2017).  \url{https://doi.org/10.1112/jlms.12009}

\bibitem{mosconi2016brezis}
Mosconi, S., Perera, K., Squassina, M., Yang, Y.:
\newblock The {B}rezis-{N}irenberg problem for the fractional {$p$}-{L}aplacian.
\newblock {  Calc. Var. Partial Differential Equations} \textbf{55}(4), 105 (2016). \url{https://doi.org/10.1007/s00526-016-1035-2}

\bibitem{mukherjee2016fractional}
Mukherjee, T., Sreenadh, K.:
\newblock Fractional {C}hoquard equation with critical nonlinearities.
\newblock {  NoDEA Nonlinear Differential Equations Appl.} \textbf{24}(6), 63 (2017). \url{https://doi.org/10.1007/s00030-017-0487-1}

\bibitem{palatucci2023nonlinear}
Palatucci, G., Piccinini, M.:
\newblock Nonlocal {H}arnack inequalities in the {H}eisenberg group.
\newblock {  Calc. Var. Partial Differential Equations} \textbf{61}(5), 185 (2022). \url{https://doi.org/10.1007/s00526-022-02301-9}

\bibitem{piccinini2022obstacle}
Piccinini, M.:
\newblock The obstacle problem and the {P}erron method for nonlinear fractional equations in the {H}eisenberg group.
\newblock {  Nonlinear Anal.} \textbf{222}, 112966 (2022). \url{https://doi.org/10.1016/j.na.2022.112966}

\bibitem{pucci2023brezis}
Pucci, P., Wang, L.:
\newblock The {B}r\'ezis-{N}irenberg equation for the {${(m,p)} $} {L}aplacian in the whole space.
\newblock {  Discrete Contin. Dyn. Syst. Ser. S.} \textbf{16}(11), 3270--3289 (2023). \url{https://doi.org/10.3934/dcdss.2023068}

\bibitem{roncal2016hardy}
Roncal, L., Thangavelu, S.:
\newblock Hardy's inequality for fractional powers of the sublaplacian on the {H}eisenberg group.
\newblock {  Adv. Math.} \textbf{302}, 106--158 (2016). \url{https://doi.org/10.1016/j.aim.2016.07.010}

\bibitem{ros2015nonexistence}
Ros-Oton, X., Serra, J.:
\newblock Nonexistence results for nonlocal equations with critical and supercritical nonlinearities.
\newblock {  Comm. Partial Differential Equations} \textbf{40}(1), 115--133 (2015). \url{https://doi.org/10.1080/03605302.2014.918144}

\bibitem{servadei2015brezis}
Servadei, R., Valdinoci, E.:
\newblock The {B}rezis-{N}irenberg result for the fractional {L}aplacian.
\newblock {  Trans. Amer. Math. Soc.} \textbf{367}(1), 67--102 (2015). \url{https://doi.org/10.1090/S0002-9947-2014-05884-4}

\bibitem{shang2014positive}
Shang, X., Zhang, J., Yang, Y.
\newblock Positive solutions of nonhomogeneous fractional {L}aplacian problem with critical exponent.
\newblock {  Commun. Pure Appl. Anal.} \textbf{13}(2), 567--584 (2014). \url{https://doi.org/10.3934/cpaa.2014.13.567}

\bibitem{wang2020properties}
Wang, X., Du, G.:
\newblock Properties of solutions to fractional {$p$}-sub{L}aplace equations on the {H}eisenberg group.
\newblock {  Bound. Value Probl.} \textbf{2020}, 128 (2020). \url{https://doi.org/10.1186/s13661-020-01425-1}

\bibitem{xiang2016existence}
Xiang, M., Zhang, B., R{\u{a}}dulescu, V. D.:
\newblock Existence of solutions for perturbed fractional {$p$}-{L}aplacian equations.
\newblock {  J. Differential Equations} \textbf{260}(2), 1392--1413 (2016). \url{https://doi.org/10.1016/j.jde.2015.09.028}

\bibitem{zhou2022kirchhoff}
Zhou, J., Guo, L., Zhang, B.:
\newblock Kirchhoff-type problems involving the fractional {$p$}-{L}aplacian on the {H}eisenberg group.
\newblock {  Rend. Circ. Mat. Palermo (2)} \textbf{71}(3), 1133--1157, (2022). \url{https://doi.org/10.1007/s12215-022-00763-6}
\end{thebibliography}
\end{document}